\documentclass{amsart}
\usepackage{amsmath,amssymb}
\usepackage[english]{babel}
\newcommand{\Gr}{Gr\"obner }

\newcommand{\In}{\textnormal{In}}

\newcommand{\St}{\mathcal{S}t_h}

\newcommand{\PP}{\mathbf{P}} 
\newcommand{\Af}{\mathbb{A}}
\newcommand{\ZZ}{\mathbb{Z}}

\newcommand{\cN}{\mathcal{N}}
\newcommand{\BSt}{B\mathcal{S}t}
\newcommand{\PS}{\mathcal{P}}

\newtheorem{lemma}{Lemma}
\newtheorem{theorem}[lemma]{Theorem} 
\newtheorem{cor}[lemma]{Corollary}

\theoremstyle{definition}
\newtheorem{example}[lemma]{Example}
\newtheorem{remark}[lemma]{Remark} 
\newtheorem{defin}[lemma]{Definition}

\begin{document}

\title{The stratum of a strongly stable ideal}

\author[M.~Roggero]{Margherita Roggero}
\address{Margherita Roggero,  Universit\`a degli Studi di Torino, Dipartimento di Matematica,  Via Carlo Alberto 10, \  10123 Torino  (Italy)}
\email{mailto:margherita.roggero@unito.it}
\urladdr{http://www2.dm.unito.it/paginepersonali/roggero/}

\thanks{Written with the support of the University Ministry funds.}

\begin{abstract}  Let $J$ be  a strongly stable monomial ideal in $\mathcal{P}=k[X_0, \dots,X_n]$ and  let $B\mathcal{S}t(J)$ be the family of all the homogeneous ideals in $\mathcal{P}$ such that the set $\mathcal{N}(J)$ of all the monomials that do not belong to  $ J$ is a $k$-vector basis of the quotient $\mathcal{P} /I$.   We show that  $I\in B\mathcal{S}t(J)$ if and only if it is generated by a special set of polynomials $G$, the $J$-\emph{basis} of $I$,  that  in some sense  generalizes the notions of Gr\"obner  and border basis (Theorem 10 and Corollary 12). Though not every $J$-basis is a Gr\"obner  basis with respect to some term ordering (Example 20),  we define two Noetherian algorithms of reduction with respect to $G$, the $G^*$-reduction (Definition 9) and the $G^{**}$-reduction (Definition 15) and prove that $J$-bases    can be characterized through a Buchberger-like criterion on the $G^{**}$-reductions of $S$-polynomials (Theorem 17). Using $J$-bases, we prove that $B\mathcal{S}t(J)$ can be endowed, in a very natural way, of a structure of affine scheme,  and that it turns out to be homogeneous with respect to a non-standard grading  over the additive group $\mathbb{Z}^{n+1}$ (Theorem 22).
 \end{abstract}
  \maketitle 
\section{Introduction}

Let $J$ be any monomial ideal in the polynomial ring $\PS:=k[X_0, \dots, X_n]$ and let us denote by $\cN (J)$  the sous-escalier of  $J$ that is the set of monomials in $\PS$ that do not belong to $J$.  In this paper we  consider  the family of all ideals $I$ in $\PS$ such that $\PS=I\oplus \langle \cN (J) \rangle $ as a $k$-vector space or, equivalently, such that $\cN (J)$ is a $k$-basis for the quotient $\PS /I$.  In this paper we investigate under which  conditions    this family is in some natural way an algebraic scheme. 

If $\cN (J)$ is a finite set, there is an evident close connection with the theory of border bases (see \cite{Mour}, \cite{ro}). 
 However, if $\cN (J)$ is not finite,  the family of such ideals can be too large. For instance if $J=(X_0)\subset k[X_0,X_1]$, the family of all ideals $I$ such that $\PS/I$ is generated by $\cN (J)=\{X_1^n, \ n\in \mathbb{N}\}$ depends on infinitely many parameters. For this reason we restrict ourselves to the homogeneous case. So, let us denote by $\BSt (J)$ the family of all \emph{homogeneous} ideals $I$ in $\PS$ such that $\PS=I\oplus \langle \cN (J) \rangle $, so that every polynomial has a unique $J$\emph{-normal form modulo} $I$ as a sum of monomials  in $\cN(J)$.
 
 Of course,        $\BSt(J)$  contains    every   ideal   $I$ such that its initial ideal  $\In_{\prec}(I)$ with respect to some term order $\prec$ is $J$, but it can also contain ideals $I$ such that $\In_\prec(I)\neq J$ for every term order $\prec$ (see for instance Example \ref{nuovo}).

Let $I$ be  an ideal in $\BSt(J)$. Using the terminology of \cite{CLO2} and \cite{SturmPolitopi},  $I $ contains a unique set of \emph{marked polynomials} $G=\{ f_\alpha= X^{\alpha} - \sum  {c}_{\alpha \gamma} X^{\gamma}\ / \ X^\alpha \in B_J\}$ where every polynomial $f_\alpha$ is given by an \emph{initial term} (or \emph{head}),   $X^\alpha$ which belongs to the monomial basis $B_J$ of $J$, and \emph{a tail} $\sum  {c}_{\alpha \gamma} X^{\gamma}$, which is a sum of   monomials   in $\cN(J)$ with constant coefficients $c_{\alpha \gamma}\in k$ ($c_{\alpha \gamma}\neq 0$ only if $\deg (X^\gamma)=\deg(X^\alpha)$). If $J=\In_\prec (I)$ for some term order $\prec$, then $G$ is indeed the reduced \Gr basis of $I$, hence it satisfies many interesting properties: it is a set of generators for $I$, it can be obtained using Buchberger algorithm, the $G$-\emph{reduction relation}   that   uses  the polynomials in $G$ as \lq\lq rewriting lows\rq\rq\  is Noetherian, and so on.

Unfortunately, if $\In_\prec(I)\neq J$ for every term order $\prec$,  the set of polynomials $G$ has not in general the good properties of  \Gr bases.  For instance  $G$ does not need to be a set of generators for $I$ (see Example \ref{es1}). Moreover, though by definition every polynomial in $\mathcal{P}$ has a $J$-normal form modulo $I$, we cannot in general obtain it through  a $ G$-reduction,  because every  sequence of $G$- reductions  on some element can lead to an infinite loop  (see Example \ref{es-8}).

We can recover in $G$ most of the good properties of \Gr bases when  $J$ is strongly stable.  For this reason we assume that $J$ is a strongly stable monomial ideal in all our statements.   As well known, in characteristic 0, this condition is equivalent to say that $J$ is Borel fixed. Note that, even under this assumption,  $G$ does not need to be a marked \Gr basis  that is a \Gr basis with respect to some term ordering (see \cite{MoraR} or \cite[page 428]{CLO2}). Hence not every sequence of  $G$-reductions  on a polynomial terminates giving   a $J$-normal form (see Example \ref{ripreso} and \cite[Theorem 3.12]{SturmPolitopi}).  However,      we  prove that   for  every polynomial $h$ in $\PS$ there is also some    $G$-reduction  giving  in a finite number of steps   the $J$-normal form modulo $I$  of $h$  (Corollary \ref{th1}) and we exhibit  two  Noetherian algorithms of reduction, the $G^*$-\emph{reduction} (Definition \ref{stella}) and the $G^{**}$-\emph{reduction} (Definition \ref{stellastella}), giving the $J$-normal form  modulo $I$ of every polynomial.  

Using the above quoted procedures of reduction,    we are able to prove     for every   strongly stable monomial ideal $J$ that $I\in \BSt(J)$ if and only if $I$ is generated by  the only set of marked polynomials $G=\{ f_\alpha= X^{\alpha} - \sum  {c}_{\alpha \gamma} X^{\gamma}\ / \ X^\alpha \in B_J\}$ it contains     (Corollary \ref{cor1}). In this case,  by analogy with \Gr basis,  we call $G$ the $J$\emph{-basis of} $I$. We also prove  that one can check whether a set of marked polynomials $G$ as above is a $J$-basis, namely  whether  it generates an ideal  $I\in\BSt(J)$,  through a Buchberger-like criterion, that is only looking at the $S$-polynomials   of elements in $G $ and to their $G^{**}$-reductions (Theorem \ref{riduzione}).  

The above quoted results allow us to extend to the family $\BSt(J)$ some properties of the \emph{\Gr stratum} $\St(J,\prec)$,  that is the family of all the ideals $I$ such that $\In_\prec(I)=J$.  In the last years several authors have been working on   $\St (J, \preceq)$,    proving that they have a natural and well defined structure of algebraic schemes, that springs out of a procedure based on Buchberger's algorithm (see \cite{CF}, \cite{NS}, \cite{RT}, \cite{FR}) and that they are homogeneous with respect to a non standard positive  grading over $\ZZ^{n+1}$ (see \cite{RT}).

Following the line  of the construction of $\St(J,\prec)$,    in the last section  of the paper we consider  for a  strongly stable ideal $J$  the   set of homogeneous polynomials 
 $\mathcal{G}=\{ F_\alpha= X^{\alpha} - \sum  {C}_{\alpha \gamma} X^{\gamma}\ / \ X^\alpha \in B_J, X^\gamma \in \cN (J)_{\vert \alpha \vert}\}$ 
where the coefficients $C_{\alpha \gamma}$ are new variables.   Imposing the above quoted  Buchberger criterion on the $\mathcal{G}^{**}$-reductions of $S$-polynomials of elements in $\mathcal{G}$,  we obtain some polynomial conditions in the   variables $C=\{C_{\alpha \gamma}\}$. The  ideal in $k[C]$  generated by these polynomials   realizes $\BSt(J)$ as an affine scheme. A possible   objection  to our construction   is that it depends on the procedure of reduction, which  is not  in general unique, and so a priori the result is  not well defined.  For this reason we first give an   intrinsic  definition of   $\BSt(J)$ as the affine scheme given by the ideal generated by minors of some matrices, and  then we show the equivalence with the one obtained through the Buchberger criterion (Theorem \ref{th2}).

By the analogy with \Gr strata, we  call $\BSt(J)$ the \emph{stratum} of $J$.  As we adopted in the two cases  a similar  construction, we can easily prove that for every fixed term ordering $\prec$,  the \Gr stratum $\St(J,\prec)$ is (scheme-theoretically) the  section of $\BSt(J)$ with a suitable linear space.

Going on  the  analogy, we consider the homogeneous structure   and  show that also    $\BSt(J)$ is homogeneous  with respect to a    grading over $\ZZ^{n+1}$   (Theorem \ref{th2}).  However, though the grading on $\St(J,\prec)$  is  positive, this property does not hold true in general for  $\BSt(J)$, so that we cannot extend to it some useful properties of \Gr strata. Especially, $\BSt(J)$ does not need to be isomorphic to an affine space when the point corresponding  to $J$ is smooth.   

In the final Example \ref{es5} we  give an explicit computation of a stratum $\BSt(J)$ which is scheme-theoretically isomorphic to an open subset of the Hilbert scheme of 8 points in $\PP^2$ (see \cite{RoggBorelCovering}). We show that it  strictly contains the family of marked \Gr bases, that is the union of the \Gr strata $\St(J,\prec)$ for every term order $\prec$, and that it is not isomorphic to an affine space, even though the point corresponding to $J$ is smooth. 

\section{Notation} \label{sec:notation}

Throughout the paper, we will consider the following general notation. 

 We work on an algebraically closed ground field $k$ of  any characteristic.
$\PS=k[X_0,\dots,X_n]$ is the polynomial ring in the set of variables $X_0,\dots,X_n$ that we will often denote by the compact notation $X$. We will denote by $X^\alpha$ any monomial in $\PS$, where $\alpha$ represents a multi-index $(\alpha_0,\dots,\alpha_n)$, that is $X^\alpha = X_0^{\alpha_0} \cdots X_n^{\alpha_n}$. $X^\alpha \mid X^\gamma$ means that $X^\alpha$ divides $X^\gamma$, that is there exists a monomial $X^\beta$ such that $X^\alpha \cdot X^\beta = X^\gamma$. 

Every ideal in $\PS$ will be homogeneous.
 If  $J$ is a monomial ideal  in $\PS$, $B_J$ will denote its  monomials basis and $\cN (J)$  its  \emph{sous-escalier}  that is the set of monomials that do not belong to $J$. We will denote by $J_m$ and $\cN( J)_m$ respectively the elements  of degree $m$ in either set.

 We fix the following  order on the set of variables  $X_0< X_1<\dots < X_n$. For every monomial $X^\alpha\neq 1$ we set $\min (X^\alpha)=\min \{ X_i \ \colon \ X_i \mid X^\alpha \}$ and $\max (X^\alpha)=\max \{ X_i \ \colon \ X_i \mid X^\alpha \}$. 

A monomial $X^\beta$ can be obtained by a monomial $X^\alpha$ through  an \emph{elementary  move} if   $X^\beta X_i=X^\alpha X_j$ for some variables $X_i\neq X_j$ or equivalently if there is a monomial $X^\delta$ such that $X^\alpha=X^\delta X_i$ and $X^\beta=X^\delta X_j$. We will say that the elementary move from $X^\alpha$ to $X^\beta$ is \emph{down} if $X_i > X_j$ and \emph{up} if  $X_i < X_j$.
 
  The transitive closure of the  elementary moves gives a quasi-order on the set of monomials of any fixed degree that we will denote by $\prec_{B}$:
$X^\alpha  \prec_{B} X^\beta $ if and only if we can pass from $X^\alpha$ to $X^\beta$ with a sequence of up elementary moves.
Note that $\prec_B$  agrees with  every term ordering $\prec$ on $\PS$ such that $X_0\prec \dots \prec X_n$, that is: $X^\alpha \prec_B X^\beta$ $\Rightarrow$  $X^\alpha \prec X^\beta$.

A monomial ideal $J \subset \PS$ is \emph{strongly stable} if and only if it contains every monomial $X^\alpha $ such that  $ X^\alpha \succ_{B} X^\beta $  and $X^\beta\in J$. If $ch (k)=0$ this  is equivalent to say that $J$ is Borel-fixed, that is   fixed under the action of the Borel subgroup of lower-triangular invertible matrices (see \cite{Eisenbud}, \S 15.9, or  \cite{Gr}).

\section{Generators of the quotient $\PS /I$ and generators of $I$}\label{sec:esempi}

The aim of this paper  is  to generalize the construction  of the homogeneous \Gr stratum $\St(J,\preceq)$ of a monomial ideal $J$ in $ \PS$ (see \cite{CF}, \cite{NS}, \cite{RT}, \cite{FR}, \cite{LR}). Roughly speaking,  $\St(J,\preceq)$ is an affine scheme that parametrizes all the homogeneous ideals  $I\subset \PS$ whose initial ideal with respect to a fixed term ordering $\preceq$ is $J$. The main tools used in the construction of this family are  reduced \Gr bases and Buchberger algorithm.  

If $B_J$ is the monomial basis of $J$, an ideal $I$ belongs to $\St (J,\preceq)$ if and only if its  reduced \Gr basis is of the type $\{ f_{\alpha} = X^{\alpha} - \sum  {c}_{\alpha \gamma} X^{\gamma}\ / \ X^\alpha \in B_J\}$ where  $c_{\alpha \gamma}\in k$,    $X^\gamma\in \cN (J)$ and  $X^\gamma \prec X^\alpha$ (of the same degree). Due to the good properties of \Gr bases, we can also observe that $\cN (J)$ is a basis of $\PS /I$ as a $k$-vector space.

 In this section we start from  the last property and explore the possibility of generalizing the above construction not making use of     any term ordering  in $\PS$. 
   
   \begin{defin} For every   monomial ideal $J$ in  $\PS$, the \emph{stratum} of $J$ is the family, that we will denote by  $\BSt(J)$, of all homogeneous ideals $I$ such that $\cN (J)$ is a basis of  the quotient $S/I$ as a $k$-vector space. 
\end{defin}

As a direct consequence of the definition, if $I\in \BSt (J)$, then we can associate to every polynomial $f$ in $\PS$   a $J$-\emph{normal form} modulo $I$, namely   a polynomial $g=\sum d_{\gamma}X^\gamma$ with $X^\gamma \in \cN(J)$ and $d_\gamma \in k$ such that $f-g\in I$. Note that if  $f$ is homogeneous, then also its $J$-normal form is,  because  $I$ is homogeneous and $\cN (J)$ is a basis of the quotient. 

Especially, we can consider the $J$-normal forms $\sum c_{\alpha \gamma}X^\gamma$  modulo $I$ of any  monomial $X^\alpha \in B_J$ and the   set of homogeneous polynomials  $G:=\{f_\alpha=X^\alpha-\sum c_{\alpha \gamma}X^\gamma \ / \ X^\alpha \in B_J\}\subset I$, that looks like  (but  does not need to be) a  \Gr basis.  
Following the terminology of \cite{SturmPolitopi}, we can consider  $G$ as a set of \emph{marked polynomials}, fixing $X^\alpha$ as  $\In(f_\alpha)$. We will call \emph{tail} of $f_\alpha$ the difference $X^\alpha -f_\alpha=\sum c_{\alpha \gamma}X^\gamma $.

 It is a natural question to ask whether any ideal $I$, containing a set of polynomials $G$ as above,   needs to belong to $\BSt(J)$.  This motivates the following:
\begin{defin}\label{defJbase} Let $J$ be a monomial ideal in $\PS$ with monomial basis  $B_J$. We will call $J$\emph{-set of marked polynomials} or simply  $J$-\emph{set}  any set of homogeneous  polynomials  of the type: 
\begin{equation}\label{eq:Jbase}
 G=\left\{ f_{\alpha} = X^{\alpha} - \sum  {c}_{\alpha \gamma} X^{\gamma} \ / \ X^\alpha \in B_J\right\} \hbox{ with }  \ X^{\gamma} \in \cN (J), \ {c}_{\alpha \gamma} \in k .
\end{equation}
Moreover we will say that $G$ is a $J$\emph{-basis} if   $\cN (J)$ is a basis of $\PS/(G)$ as a $k$-vector space.
\end{defin}

 It is quite obvious from the definition, that the ideal generated by a $J$-basis has the same Hilbert function than $J$.

\medskip

The following examples shows that  not every $J$-set is also a $J$-basis and  that, more generally,     $J$-sets or even  $J$-bases do not have the good properties of \Gr bases.  Moreover, we will see that not every $J$-basis, as a marked set, is also a \Gr basis with respect to some term ordering, even if $J$ is a strongly stable monomial ideal.

\begin{example}\label{es1} In $k[x,y,z]$ let $J=(xy,z^2)$ and $I=(f_1=xy+yz, f_2=z^2+xz)$. The ideal $I$ is generated by a $J$-set. However $J$ defines a   $0$-dimensional subscheme  in $\PP^2$, while $I$ defines a $1$-dimensional subscheme (because it contains the line $x+z=0$). Therefore, $\{ f_1,f_2 \}$ is not a $J$-basis because $I$ and $J$ do not have the same Hilbert function.
\end{example}

Even when  the ideal $I$ is generated by a  $J$-set and $I$ and $J$ share the same  Hilbert function,  the $J$-set is not necessarily a $J$-basis, that is   $\cN (J)$ does not need to be a basis for $\PS /I$ as a $k$-vector space. 

\begin{example}\label{es2}  In $k[x,y,z]$, let  $J=(xy,z^2)$ and let $I$ be the ideal generated by the $J$-set $\{g_1=xy+x^2-yz \ , \  g_2=z^2+y^2-xz\}$. It is easy to verify that both $J$ and $I$ are  complete intersections of two quadrics  and then they have the same Hilbert function. However, $I\notin \BSt(J)$ because $\cN (J)$ is not free in  $k[x,y,z]/I$: in fact $zg_1+yg_2=x^2z+y^3 \in I$ is a sum of monomials in $\cN (J)$. Hence $\{g_1,g_2\}$ is not a $J$-basis.
\end{example}

When the monomial ideal $J$ is the initial ideal of $I$ with respect to a term ordering, then $J$ is a basis of $\PS/I$ as a $k$-vectors space and the reduced \Gr basis of $I$ is indeed  the $J$-set contained in $I$ and it is also  $J$-basis for $I$. Hence $I$ has a $J$-basis which acts also as a set of generators.    However in general  the converse of these properties does not hold.  In fact even if $\cN(J)$ is a $k$-basis for $\PS/I$, the   $J$-set $G$ contained in $I$ does not need to be a $J$-basis because does not need to be a set of generators for  $I$. Furthermore, not every $J$-basis as a marked set is a \Gr basis with respect to a suitable term ordering. 

\begin{example}\label{es3} In $k[x,y,z]$ let $J=(xy,z^2)$ and $I=(f_1=xy+yz, f_2=z^2+xz, f_3=xyz)$. Both $I$ and $J$ define $0$ dimensional subschemes in $\PP^2$ of degree $4$ and in fact $I\in \BSt(J)$. In order to verify that $\cN (J)$ is a basis for $k[x,y,z]/I$ it is sufficient to show that, for every $m\geq 2$,  the $k$-vector space $V_m=I_m+\cN (J)_m=I_m+\langle x^m,y^m,x^{m-1}z, y^{m-1}z \rangle$ is equal to $k[x,y,z]_m$. For $m=2$, this is   obvious. Then, assume $m\geq 3$. First of all we observe that  $yz^2=zf_1-f_3\in I$: then $V_m$ contains all the monomials $y^{m-i}z^i$. Moreover  $x^2y=xf_1-f_3\in I$ and $xy^{m-1}=y^{m-2}f_1-zy^{m-1}\in V_m$: then $V_m$  contains all the monomials $x^{m-i}y^i$. Finally, we can see by induction on $i$ that all the monomials $x^{i}z^{m-i}$ belong to $V_m$. In fact as already proved,   $z^m\in V_m$, hence    $x^{i-1}z^{m-i+1}\in V_m$ implies $x^{i}z^{m-i}=x^{i-1}z^{m-i-1}f_2-x^{i-1}z^{m-i+1}\in V_m$. 

 However the  $J$-set $\{f_1,f_2\}$ is not a $J$-basis  because it does not generate $I$ (see Example \ref{es1}). 
\end{example}

\begin{example}\label{nuovo}  Let us consider in $k[x,y,z]$    the strongly stable monomial ideal $J$ generated by the set $B_J$ of all the degree 5 monomials in $(x^3,x^2y,xy^2,y^5)$  and    the ideal  $I$ generated by the marked set $G=B_J\cup \{f\}\setminus \{xy^2z^2\}$, where $f=xy^2z^2-y^4z-x^2z^3$ and $\In(f)=xy^2z^2$. As $J$ is strongly stable, the  results obtained  in  following sections will allow us to show that $G$ is the $J$-basis of $I$ (see Example \ref{ripreso}). However, the marked set $G$ is not a \Gr basis of $I$ with respect to any term ordering $\prec$, because  $xy^2z^2\succ y^4z$ and $xy^2z^2 \succ x^2z^3$  would be in contradiction with the equality  $(xy^2z^2)^2 =x^2z^3\cdot y^4z$.
\end{example}

\section{Characterizations of $J$-bases}
\textbf{From now on, $J$ will always denote a strongly stable monomial ideal in $\PS$.}

Under this assumption  we will see that  we can recover in $J$-bases most of the good properties of \Gr bases (though the two notions are not equivalent neither under this stronger condition on $J$, as shown by the above Example \ref{nuovo}). 
First of all we consider   properties that concern the reduction relation.
\begin{defin} Let  $G$ be a set of marked polynomials  
\begin{equation}\label{def:G} G=\{f_\alpha=X^\alpha-\sum c_{\alpha \gamma}X^\gamma \ / \ \In(f_\alpha)=X^\alpha \}.
\end{equation}
 A $G$\emph{-reduction} of a polynomial $h$, denoted by $h \stackrel{G}\longrightarrow h_0$ is a sequence:
$$h=h_0 \stackrel{G}\longrightarrow h_1 \stackrel{G}\longrightarrow  \dots \stackrel{G}\longrightarrow h_i\stackrel{G}\longrightarrow h_{i+1} \stackrel{G}\longrightarrow \dots \stackrel{G}\longrightarrow h_s$$
where $h_i\in\PS$ and each step $h_i\stackrel{G}\longrightarrow h_{i+1}$ is obtained  rewriting   a monomial $X^\beta X^\alpha$ ($X^\alpha =\In(f_\alpha)$) that appears in $h_i$    by $X^\beta \left(\sum c_{\alpha \gamma}X^\gamma \right)$. 
We say that such a sequence   is \emph{complete} and write $h\stackrel{G}\longrightarrow_+ h_s$ if   $h_s$   is $G$\emph{-reduced} that is does not contain  monomials multiple of some $X^\alpha$. 
\end{defin}

In \cite[Theorem 3.12]{SturmPolitopi}  \Gr bases are characterized in terms of reductions: there exists a term ordering such that $G$    is a \Gr basis  w.r.t. it if and only if the $G$-reduction is Noetherian, that is every sequence of $G$-reductions terminates in a finite number of steps.  Thus, if $G$ is not a \Gr basis, we can find some polynomial $h\in \PS$ and some procedure of $G$-reduction of $h$  that does not terminate in a finite number of steps. In the following example we exhibit a marked set  $G$ and a polynomial    $h$  such that $h$ has no complete reductions that is such that every $G$-reduction of $h$ does not terminate. 

\begin{example}\label{es-8} Let us consider the set of marked polynomials $G=\{  f_1=xy+yz, f_2=z^2+xz \}$ with $\In(f_1)=xy$ and $\In(f_2)=z^2$. The only possible first step of $G$-reduction of the monomial $h=xyz$ is $xyz-zf_1=-yz^2$. The only possible first step of $G$-reduction of $-yz^2$  is   $-yz^2+yf_2=xyz$, which is again the initial monomial. Observe that the ideal generated by the initial monomials $xy$ and $z^2$   is not strongly stable.
\end{example}
 Now we will prove  that  the  situation illustrated by the previous example is not allowed if we assume that  the initial monomials of the elements of $G$ generate a   strongly stable ideal, and present an algorithm giving  a complete $G$-reduction of every polynomial.

 \begin{defin}\label{stella} Let $J$ be a strongly stable monomial ideal with basis $B_J$ and $G$ be a set of marked polynomials: 
 \begin{equation}\label{eq5} G=\{f_\alpha=X^\alpha - \sum c_{\alpha \gamma} X^\gamma  \ / \ B_J=\{ \In(f_\alpha)=X^\alpha \} \}.\end{equation} 
  We define the  $G^*$\emph{-reduction} of a polynomial $h$, denoted $h\stackrel{G^*}\longrightarrow h_0$ as a special $G$-reduction which is a sequence of steps obtained in the following way. We proceed  by induction on the degree $m$  and   assume that every polynomial $g$ of degree $m-1$ either is in $J$-normal form or has a \emph{complete} $G^*$-\emph{reduction }  $g\stackrel{G}\longrightarrow_+  g_0$ that is a $G^*$-reduction  to a $J$-normal form $g_0$.  Assume that  $h$ has degree $m$ and is not in $J$-normal form. Then we fix a monomial  $X^\beta \in J$   that appears in $h$: 

  if $X^\beta=X^\alpha \in B_J$, we rewrite it using $f_\alpha\in G$; 

 if  $X^\beta \notin B_J$, that is $X^\beta=X_i X^{\delta}$ for some  $X^\delta \in J_{m-1}$,  we rewrite it using $X_ig_0$ where $X^\delta \stackrel{G}\longrightarrow _+ g_0$.
\end{defin}

\begin{theorem}\label{G*termina} Let  $J$ be a  strongly stable monomial ideal in $\PS$ with  monomial basis $B_J$ and let $G$ be as in (\ref{eq5}). Then the  $G^*$-reduction is Noetherian, that is every sequence of $G^*$-reductions of a polynomial $h$ terminates in a finite number of steps leading to a $J$-normal form $h_0$.
\end{theorem}
\begin{proof} It is sufficient to prove that our assertion holds for the monomials. Let us consider the set $E$ of monomials having some $G^*$-reduction that does not terminate. If $E\neq \emptyset$ and $X^\beta in E$, then of course   $X^\beta \in J\setminus B_J$, because for every 
   $ X^\alpha \in B_J$,   the only possible    $G^*$-reduction is $X^\alpha \stackrel{G^*}\longrightarrow \sum c_{\alpha \gamma} X^\gamma$  which is complete by hypothesis. Thus $X^\beta=X_iX^{\delta}$ for some  $X^\delta  \in J$: we choose  $X^\beta$ so that its degree $m$ is the minimal in $E$   and that, among the monomials of degree $m$ in $E$,   $X_i$ is minimal with respect to $\prec_B$.  The first step of $G^*$-reduction of $X_iX^\delta  $ is $X_iX^\delta \stackrel{G^*}\longrightarrow X_i g_0$,  where the complete reduction $X^\delta \stackrel{G^*}\longrightarrow _+ g_0$  exists because $X^\delta$ has degree $m-1$. 
 Even if $g_0$ is a $J$-normal form,   $X_ig_0$  does not need to be. If  $X_ig_0$ contains a monomial    $X_iX^\gamma=X_jX^\beta$ for some $X^\beta \in J$, then $X_j\prec_B X_i$ because  $J$ is strongly stable. 
 By the minimality of $X_i$, every  sequence of $G^*$-reductions on monomials of $X_ig_0$ terminates. This is a contradiction and  so $E$ is empty.
\end{proof}

Using the $G^*$-reductions we can generalize   some properties of \Gr bases.

\begin{cor}\label{th1} Let  $J$ be a  strongly stable monomial ideal in $\PS$ with  monomial basis $B_J$ and let $I$  be a homogeneous ideal that contains  a $J$-set $G$ as in (\ref{eq5}). Then:
 \begin{itemize}
\item[i)] every polynomial $h$ in $\PS$ has a normal form modulo $I$ as a sum of monomials in $\cN(J)$     that can be obtained through a $G^*$-reduction.
\item[ii)]  $\cN (J)$ generates  $\PS/I$ as a $k$-vector space, so that $\dim_k I_m \geq \dim_k J_m$ for every $m\geq 0$. 
\end{itemize}
\end{cor}

\begin{proof}   The first item is a straightforward consequence of Theorem \ref{G*termina}: in fact at every step of $G$-reduction, and so especially of $G^*$-reduction, we add a linear combination of polynomials $f_\alpha \in G\subset I$. Then for every $h\in \PS$, we can find $h \stackrel{G^*}\longrightarrow_+ h_0$, where $h_0$ is a $J$-normal form and $h-h_0\in I$.
 
 The second item is an obvious consequence of the first one.
\end{proof}

The following result, which is again a direct consequence of the existence of a Noetherian reduction,  clearly highlights the analogy  between $J$-bases and \Gr or border bases.

\begin{cor}\label{cor1} Let  $J$ be a  strongly stable monomial ideal in $\PS$  and let $I$  be a homogeneous ideal that contains  a $J$-set $G$.  The following are equivalent:
\begin{enumerate}
	\item $I\in \BSt(J)$;
	\item $G$ is a $J$-basis;
	\item $\cN(J)$ is free in $\PS /I$;
	\item $I$ and $J$ share the same Hilbert function, i.d. $\forall m$:  $\dim_k I_m=\dim_k J_m$;
		\item  $\forall m$:  $\dim_k I_m\leq \dim_k J_m$;
	\item $\forall h\in I$:    $h \stackrel{G^*}\longrightarrow_+ 0$;
	\item  if $h_0$ is  any  $J$-normal form  modulo $I$ of a polynomial $h\in \PS$, then   $h \stackrel{G^*}\longrightarrow_+ h_0$;

	\item 	 if $h_0$ is  any $J$-normal form  modulo $I$ of a monomial $X^\beta \in\PS$, then   $X^\beta \stackrel{G^*}\longrightarrow_+ h_0$;
\end{enumerate}
\end{cor}
\begin{proof} (1), (2) and (3) are equivalent by the definition itself of $J$-basis. (4) and (5) are   equivalent to the previous ones by Corollary \ref{th1} ii).  For (6) $\Leftrightarrow$ (7) $\Leftrightarrow$  (8) it is sufficient to observe that if $h_0$ is  a sum of monomials in $\cN(J)$, then $h \stackrel{G^*}\longrightarrow_+ 0$ if and only if  $h-h_0 \stackrel{G^*}\longrightarrow_+ 0$.

 Finally, to prove  (3)$\Leftrightarrow$ (6) we observe that (3) says that the only normal form of an element of $ I$ is $0$ and that, by Corollary \ref{th1} i), there is a $G^*$-reduction from every polynomial to one of its normal forms.
\end{proof}
\section{Detecting $J$-bases}\label{Buch}

In the  following $B_J$   will always denote the monomial basis of    a strongly stable monomial ideal $J$ and $I$ the ideal generated by a set of marked polynomials $G$ as in (\ref{eq5}), i.e. generated by a $J$-set.

 Now we will complete the analogy with the case of \Gr bases, showing that one can check if   $I$   belongs to $\BSt(J)$ using a Buchberger-like criterion that is looking at   $G$-reductions of $S$-polynomials. Mimicking the terminology of the \Gr bases,  we will call   \emph{$S$-polynomial} of  $f_\alpha, f_{\alpha'} \in G$    the polynomial $S(f_\alpha, f_{\alpha'}) =X^{\beta}f_{\alpha}-X^{\beta'}f_{\alpha'}$, where $X^{\beta+\alpha}=X^{\beta'+\alpha'}=LCM(X^{\alpha},X^{\alpha'})$.  
 
 We will reduce the $S$-polynomials using  an algorithm of $G$-reduction, denoted by  $G^{**}$ because it is in fact a refinement of the $G^*$-reduction and whose  definition requires some preliminaries.

In every degree $m$ the vector space  $I_m$ is generated by the set of polynomials $W_m=\{X^\delta f_\alpha / \ X^{\delta+\alpha}\in J_m, \ X^\alpha \in B_J\}$. $W_m$ becomes  a set of marked polynomials and also a partially ordered set in the following  way.

\begin{defin}\label{minimali}  The \emph{initial monomial} or \emph{head } of  $X^\delta f_\alpha$ is $\In(X^\delta f_\alpha)= X^{\delta+\alpha}$   and its \emph{tail} is the difference $X^{\delta+\alpha}-X^\delta f_\alpha=\sum c_{\alpha \gamma}X^{\delta +\gamma}$.   If $f_\alpha, f_{\alpha'}\in W_m$:
\begin{equation}\label{ordine} X^\delta f_\alpha > X^{\delta'} f_{\alpha'} \Longleftrightarrow \hbox{ the first non-zero entry in } \delta' -\delta  \hbox{  is positive}.\end{equation}
\end{defin}

\begin{lemma}\label{lemma1} Let $X^\beta$ be any monomial in $J_m$ and let  $W_{\beta}$ be the subset of   $W_m$ containing   all the monomials with   head   $X^\beta$.

Then $W_\beta$ is totally ordered by $<$  and its minimum is the only element $X^\delta f_\alpha\in W_\beta$   such that either $X^\delta=1$ or $\max(X^\delta) \leq \min(X^\alpha)$.
\end{lemma}
\begin{proof}   For the first assertion it is sufficient to observe that for every two different elements $X^\delta f_\alpha$, and $ X^{\delta'} f_{\alpha'}$ in $ W_\beta$ the equality  $X^\beta=X^{\alpha+\delta}=X^{\alpha'+\delta'}$   implies $\delta-\delta'\neq 0$. In fact if $\delta=\delta'$, then   $X^{\alpha}=X^{\alpha'}$, hence by definition of $G$ also  $f_{\alpha}=f_{\alpha'}$. Thus, $W_\beta$ has minimum because it is a finite, totally ordered set.

If $X^\delta=1$, namely if $X^\beta=X^\alpha \in J_m$, then    $f_\beta$  is the only element in $W_\beta$ because $B_J$ is a monomial basis.
 So assume that $X^\delta \neq 1$ and let $X^\delta f_\alpha $ be an element in $ W_\beta$   such that $X_i=\max (X^\delta) > X_j=\min(X^\alpha )$. Now we verify  that  $X^\delta f_\alpha $ is not the minimum of $ W_\beta$. Since $J$ is Borel fixed, then  $X_iX^\alpha/X_j \in J_m$. Hence $W_\beta$ also contains  $X^{\delta'+\eta}f_{\alpha'}$  where $X^{\delta'}=X_jX^\delta/X_i$ and $X_iX^\alpha/X_i=X^\eta X^{\alpha'}$. The first non-zero entry of $\delta'+\eta-\delta$ is  positive because the same holds for   $\delta'-\delta$.
\end{proof}  

Thanks to the previous lemma we can consider for every $m$ 	 the subset $V_m\subseteq W_m$ of the polynomials that are minimal with respect to $<$: note that  $V_m$ has $\dim_k (J_m)$ elements.

 \begin{defin}\label{stellastella} The $G^{**}$-reduction of a homogeneous polynomial $h$ of degree $m$ is a $G$-reduction that only uses the polynomials of $V_m$, that is in which at every step we rewrite a monomial of $ J_m$ by means of  a  polynomial of $ V_m$. 
 \end{defin}
 
  Note that when $X^\delta \neq 1$,  the monomials in the tail of $X^\delta f_\alpha$ do  not need to belong to  $\cN(J)$. 
 \begin{lemma}\label{ridmonomi} In the above hypotheses and notation:
 \begin{itemize}\item[a)]  if  $ X^\delta f_\alpha \in W_m$, $X^{\delta'} f_{\alpha'}\in V_m$  and   $X^{\delta'+\alpha'}$ appears   in the tail of $X^\delta f_\alpha$, then $ X^\delta f_\alpha >X^{\delta'} f_{\alpha'}$.

\item[b)]  if $X^{\eta_0},X^{\eta_1} \dots, X^{\eta_s}$ is a sequence of monomials in $J_m$ such that $X^{\eta_{i+1}}$ appears in the tail of the polynomial $g_{\eta_i}=\min(W_{\eta_i})$,
then   the monomials $X^{\eta_i}$ are all distinct;

\item[c)] every polynomial $h$ has a $G^{**}$-reduction, unless   it is   a sum of monomials in $\cN(J)$, hence  $h \stackrel{G^{**}}\longrightarrow_+ h_0$ only if $h_0$ is a $J$-normal form modulo $I$ of $h$;

\item[d)] the $G^{**}$-reduction is Noetherian.
\end{itemize}
\end{lemma}
\begin{proof}
  a) Every monomial in the tail of $X^\delta f_\alpha$ is a multiple of $X^\delta$, and so especially   $X^{\delta'+\alpha'}=X^{\delta +\gamma} $ for some $X^\gamma \in \cN(J)$. By Lemma \ref{lemma1},  the first non-zero entry  in $\delta' -\delta$ is positive and so $X^{\delta'} f_{\alpha'} < X^{\delta} f_{\alpha}$.

\medskip

 b) As a consequence of  a) we have $g_{\eta_1}> g_{\eta_2} >\dots >g_{\eta_s}$ so that the $g_{\eta_i}$ are all distinct. Thus also their heads  $X^{\eta_1} $  are distinct because each  $g_{\eta_i}$ is the minimum in $W_{\eta_i}$. 

 For c)  it is sufficient to observe that, by construction, every monomial $X^\eta$ in $J_m$ is the head of one and only one polynomial $g_\eta$ in $V_m$ and so $X^\eta \stackrel{G^{**}}\longrightarrow X^\eta-g_\eta$.
 
  d) is a consequence of c) and b). In fact    the length of every sequence of monomials as in b) is bounded by the number  of monomials of degree $m$, while from a  sequence of  infinitely many steps of $G^{**}$-reduction   we could obtain  sequences of monomials as in b) as long as we want.
\end{proof}

Now we are able to prove the  \emph{Buchberger-like criterion for $J$-bases}.

\begin{theorem}\label{riduzione}  Let  $J$ be a  strongly stable monomial ideal in $\PS$ with  monomial basis $B_J$ and let $I$ be a homogeneous ideal generated by a $J$-set $G$ as in (\ref{eq5}). Then:
$$I\in \BSt(J)  \Longleftrightarrow  \forall    f_\alpha, f_{\alpha'}\in G:  \  S(f_\alpha, f_{\alpha'})\stackrel{G^{**}}\longrightarrow_+ 0.$$
\end{theorem}

\begin{proof}   
Due to Lemma \ref{ridmonomi},  every polynomial $h$ has a complete $G^{**}$-reduction  to a $J$-normal form modulo $I$.  Moreover, if $I\in \BSt(J)$ the normal form is unique (Theorem \ref{th1}) and the unique normal form of every polynomial in $I$, as the  $S$-polynomials  are,   is 0.

 For the converse,  we use \lq\lq(1) $\Leftrightarrow$ (5)\rq\rq\ in Corollary \ref{cor1}  and prove that, for every $m$, the $k$-vector space $I_m$ is generated by the $\dim_k (J_m)$ elements of $V_m$. More precisely we will show that     every  polynomial $X^\delta f_\alpha \in W_m$ either   belongs to $V_m$ or   is   a linear combination of  elements of $V_m$ lower than $X^\delta f_\alpha$ itself. We may assume  that the same holds for every polynomial in $W_m$ lower than $X^\delta f_\alpha$.
 
  If  $X^\delta f_\alpha \in V_m$ there is nothing to prove. 
  On the other hand, let $X^{\delta'}f_{\alpha'}=\min(W_{\delta+\alpha})\in V_m$ and consider $X^{\delta}f_{\alpha}-X^{\delta'}f_{\alpha'}$.  
 If this difference is the $S$-polynomial $S( f_{\alpha}, f_{\alpha'})$, by the hypothesis it has a complete  $G^{**}$-reduced to 0,  hence it is a $k$-linear combination $\sum z_ig_{\eta_i}$ of polynomials   $g_{\eta_i}V_m$ with constant coefficients $z_i\in k$.   Thanks to Lemma \ref{ridmonomi} a), we have for each $i$ either  $X^{\delta}f_{\alpha}>g_{\eta_i}$  or  $X^{\delta'}f_{\alpha'}> g_{\eta_i}$ and so again $X^{\delta}f_{\alpha}>g_{\eta_i}$ as $ X^{\delta'}f_{\alpha'}\in V_m$. Thus $X^{\delta} f_{\alpha}$ is a linear combination $ X^{\delta'}f_{\alpha'}+\sum z_ig_{\eta_i}$ of elements in $V_m$ lower than it.
 
If $X^{\delta}f_{\alpha}-X^{\delta'}f_{\alpha'}=X^\beta S( f_{\alpha}, f_{\alpha'})= X^\beta (X^{\eta}f_{\alpha}-X^{\eta'}f_{\alpha'})$ for some $X^\beta \neq 1$, we can apply the previous case to $S( f_{\alpha}, f_{\alpha'})$ and say that $X^{\eta}f_{\alpha}$ is a sum $X^{\eta'}f_{\alpha'}  +\sum z_ig_{\eta_i}$ of polynomials in $V_{m-\vert \beta \vert}$ lower than $X^{\eta}f_{\alpha}$. Hence $X^{\delta}f_{\alpha}=X^{\beta+\eta'}f_{\alpha'}  +\sum z_i X^{\beta'}g_{\eta_i}$. The polynomials $X^{\beta+\eta'}f_{\alpha'}$ and $  X^{\beta'}g_{\eta_i}$ appearing in the right hand are lower that  $X^{\delta}f_{\alpha}$. So we can apply to them the inductive hypothesis and say that either they are elements of $V_m$ or they are  linear combinations of lower elements in $V_m$. This allows us to conclude.
\end{proof}  

\begin{remark}\label{solom0} In the proof of \lq\lq$\Longleftarrow$\rq\rq\ of the previous theorem we do not use in fact the whole hypothesis, that is the existence of a $G^{**}$-reduction for all the $S$-polynomials $X^\delta f_{\alpha}-X^{\delta'} f_{\alpha'}$ of elements in $G$, but only   for those such that either $ f_{\alpha}$ or $ f_{\alpha'}$ belongs to some $V_m$.
In the case of \Gr basis, this  property  is analogous  to the improved Buchberger algorithm that only considers $S$-polynomials corresponding to a set of generators for the syzygies of $J$.

Thus we can improve Corollary \ref{cor1} and say that, in the same hypotheses: 
$$
I \in BSt(J) \Longleftrightarrow  \forall m\leq m_0: \  \dim_k I_m=\dim_k J_m \Longleftrightarrow  \forall m\leq m_0: \  \dim_k I_m\leq \dim_k J_m
$$
where $m_0$ is the maximal degree of $S$-polynomials of the above special type. 

Moreover,   to prove that $\dim_k I_m=\dim_k J_m $ for some $m$ it is sufficient to assume that there exists a $G^{**}$-reduction to 0 of the $S$-polynomials of degree $\leq m$.
\end{remark}

\begin{example} Let $J=(x^2,xy,xz,y^2)\subset k[x,y,z]$, where $x>y>z$ and consider a $J$-set $G=\{f_{x^2},f_{xy}, f_{xz}, f_{y^2}\}$. In order to check whether $G$ is a $J$-basis it is sufficient to verify if $S(f_{x^2},f_{xy})$, $S(f_{x^2},f_{xz})$, $S(f_{x^2},f_{y^2})$, $S(f_{xy},f_{xz})$ and $S(f_{xy},f_{y^2})$ have  $G^{**}$-reductions to $0$, but it is not necessary to controll $S(f_{xz},f_{y^2})$ because the element in $V_3$ whose head is $xy^2z$ is $yzf_{xy}$.
\end{example}

\begin{example}\label{ripreso}  Let us consider   again, as in Example \ref{nuovo},  the  strongly stable ideals $J=(x^3,x^2y,xy^2,y^5)_{\geq 5}$   in $ k[x,y,z]$  and  the marked set $G=B_J\cup \{f\}\setminus \{xy^2z^2\}$, where $f=xy^2z^2-y^4z-x^2z^3$   and $\In(f)=xy^2z^2$. We have already proved that $G$ is not a \Gr basis with respect to any term ordering. However, it is a $J$-basis as we can verify using the Buchberger like criterion proved in Theorem \ref{riduzione}.  The $S$-polynomials non involving $f$ vanish and all the $S$-polynomials involving $f$ are multiple of either $x\cdot(y^4z+x^2z^3)$ or $y\cdot(y^4z+x^2z^3)$.    Since  $y^4z\cdot x$, $y^4z\cdot y$, $x^2z^3\cdot x$,  $ x^2z^3\cdot y$ belong to  $B_J\setminus\{xy^2z^2\}$ the $G^{**}$-reduction of all the $S$-polynomials  is  0. Notice that the  $G^{**}$-reduction of $x^2y^2z^3 $  is  0,  because $z^2\cdot x^2y^2z \in V_7$, while $xzf \notin V_7$. A different choice of $G$-reduction  gives the loop:
$$x^2y^2z^3 \stackrel{f}\longrightarrow xy^4z^2 +x^3z^4  \stackrel{ x^3z^2}\longrightarrow  xy^4z^2  \stackrel{f}\longrightarrow y^6z+x^2y^2z^3 \stackrel{ y^5}\longrightarrow x^2y^2z^3$$
\end{example}

 \section{The stratum of $J$ as an homogeneous affine scheme}\label{sec:local}
 
  In this last section, using the characterizations of $J$-bases obtained in the previous ones,  we finally prove that the family $\BSt(J)$ can be endowed  in a natural  way of a structure of  affine scheme and that it turns out to be     homogeneous with respect to a grading over $\ZZ^{n+1}$. Our construction generalizes that of the \Gr stratum $\St(J,\prec)$ given in \cite{CF}, \cite{NS}, \cite{RT}, \cite{FR} and \cite{LR}. Especially as in \cite{LR} we obtain equations giving the affine structure    in two   equivalent  ways, that is using either reductions (i.e. the Buchberger criterion)  or the rank of some matrices.  
  
  The main difference between the cases of $\St(J,\prec)$ and  $\BSt(J)$ is that in the first one $J$ is any monomial ideal, but we need to fix a term ordering, while in the second one we do not need any term ordering, but $J$ is assumed to be strongly stable. 
  
 Let us fix any strongly stable monomial ideal $J$ in $\PS=k[X_0, \dots, X_n]$ and let $B_J$ be its monomial basis.  Summarizing what proved in the first part of this paper, every ideal $I\in \BSt(J)$ has a unique set of generators of the type (\ref{eq5}), that is a  set of marked polynomials of the type:
 $$G=\{f_\alpha=X^\alpha - \sum c_{\alpha \gamma} X^\gamma  \ / \ \In(f_\alpha)=X^\alpha\in B_J, \ X^\gamma \in \cN(J)_{\vert \alpha \vert} \}.
 $$
 We  give an explicit description of  the family $\BSt(J)$ using this special set of generators. 
Let us consider a new variable $C_{\alpha \gamma}$ for every monomial $X^\alpha$ in  $ B_J$  and for every monomial $X^\gamma \in \cN (J)$ of the same degree as $X^\alpha$. We  denote by $C$ the set of these new variables and by $N$ their number. 

Moreover,  let us consider the set of marked polynomials :
\begin{equation}\label{generale} \mathcal{G}=\{F_\alpha=X^\alpha-\sum C_{\alpha \gamma} X^\gamma  \ / \ \In(F_\alpha)=X^\alpha\in B_J, \ X^\gamma \in \cN(J)_{\vert \alpha \vert}\}.
\end{equation}  
We can obtain the $J$-basis of every ideal $I\in \BSt(J)$   specializing (in a unique way) the variables $C$ in $k^N$, but not every specialization gives rise to an ideal in $\BSt(J)$. 
Thanks to Corollary \ref{cor1}, a specialization $C\mapsto c\in k^N$ transforms $\mathcal{G}$ in a $J$-basis  that generates an ideal   $I \in \BSt(J)$  if and only if   $\dim_k I_m \leq \dim_kJ_m$ for every $m\geq 0$.  This condition     realizes $\BSt(J)$ as an affine subscheme of $\Af^N$.

 We can then consider for each $m$ the matrix $A_m$ whose columns correspond to the degree $m$ monomials in $\PS=k[X]$ and whose rows contain the coefficients of those monomials  in every polynomial of the type $X^\beta F_\alpha$ such that  $\vert \beta +\alpha\vert =m$:  every entry in $A_m$ is either 0, 1 or  one of the variables $C$.

 We will denote by   $\mathfrak{A}$ the ideal of $k[C]$   generated by all the minors of  $A_m$ of order $\dim_k(J_m)+1$ for every $m\geq 0$.

 \medskip
 
 We can also consider the $S$-polynomials  $S(F_\alpha,F_{\alpha'})$ of elements in $\mathcal{G}$ and their $\mathcal{G}^{**}$-complete reductions $H_{\alpha,\alpha'}$. Then we collect $H_{\alpha,\alpha'}$   with respect to the variables $X$ and extract their   coefficients  that  are polynomials in $k[C]$. We will denote by $\mathfrak{R}$ the ideal in $k[C]$ generated by the set of these coefficients. Moreover we will also denote by $\mathfrak{R}'$ the ideal   obtained in this same way, but only considering $S$-polynomials  $S(F_\alpha,F_{\alpha'})=X^\delta F_\alpha-X^{\delta'}F_{\alpha'}$ such that $X^\delta F_\alpha$ is minimal among those with head $X^{\delta+\alpha}$.

\medskip

We   finally get  the main purpose of the paper, that is we  define   the algebraic structure of the stratum of $J$ and prove some general properties.

\begin{defin} In the above settings the stratum $\BSt(J)$ of $J$ is the subscheme of $\Af^N$ defined by the ideal $\mathfrak{A}$.
\end{defin}
\begin{theorem} \label{th2} Let  $J$ be a strongly stable ideal in $\PS$. In the above hypotheses and notation:
\begin{itemize} 
\item[1)]  there is an $1-1$ correspondence between the closed points of $\BSt(J)$and  the ideals $I\in \PS$ such that $\cN(J)$ is a  basis of the $k$-vector space $\PS/I$. One of the closed points of $\BSt(J)$ is the origin,  which   corresponds to the ideal $J$ itself. 
\item[2]  $\BSt (J)$ is  homogeneous with respect to a non-standard grading $\lambda$ of $k[C]$ over the group $\ZZ^{n+1}$ given by   $\lambda (C_{\alpha \gamma})=\alpha-\gamma$.
\item[3]  $\mathfrak{A}=\mathfrak{R}= \mathfrak{R}'$, hence $\BSt(J)$ can also be defined using  the  Buchberger like criterion about  $\mathcal{G}^{**}$-reductions of $S$-polynomials of elements in $\mathcal{G}$. 
\end{itemize}

\end{theorem}
\begin{proof}   
 1) is a consequence of what proved in the previous sections.

To prove that $\BSt(J)$ is $\lambda$-homogeneous it is sufficient to show that every minor of $A_m$ is $\lambda$-homogeneous. Let us  denote by $C_{\alpha \alpha}$ the coefficient ($=1$) of $X^\alpha$ in every polynomial $F_\alpha$: we can apply also to the \lq\lq symbol\rq\rq\ $C_{\alpha \alpha}$ the  definition of $\lambda$-degree of the variables $C_{\alpha \gamma}$,  because $\alpha-\alpha=0$ is indeed the $\lambda$-degree of the constant $1$. In this way, the  entry in the row $X^\beta F_\alpha$ and in the column $X^\delta$  is  $\pm C_{\alpha \gamma}$ if  $X^{\delta}=X^{\beta}X^{\gamma}$ and is 0 otherwise. 

Let us  consider in  a matrix $A_m$ the minor  corresponding to some rows $X^{\beta_i}F_{\alpha_i}$ and to some columns $X^{\delta_j}$, $i,j=1,\dots, s$.   Every monomial that appears in such a minor is of the type $\prod_{i=1}^s C_{\alpha_i \gamma_{j_i}} $ where $\{j_1, \dots, j_s\}=\{1, \dots, s\}$ and $X^{\delta_{j_i}}=X^{\beta_i}X^{\gamma_{j_i}}$. Then its degree is:
$$\sum_{i=1}^s \left( \alpha_i -\gamma_{j_i}\right) =\sum_{i=1}^s \left( \alpha_i -\delta_{j_i}+\beta_i\right) =\sum_{i=1}^s \left( \alpha_i +\beta_{i}\right) -\sum_{j=1}^s\delta_{j}$$
 which only depends on the minor.
 
 \medskip
 
  3) Let $a_m=dim_kJ_m$. We consider in $A_m$ the $a_m \times a_m$ submatrix  whose columns corresponds to the minors in $J_m$ and whose columns are given  by the polynomials $X^\beta F_\alpha$ that are minimal with respect to the partial order given in Definition \ref{minimali}. Up to a permutation  this submatrix is upper-triangular with 1 on the main diagonal. We may also assume that it corresponds to the first $a_m$ rows and columns in $A_m$.   Then $\mathfrak{A}$ is generated by the determinants of $a_m+1\times a_m+1$ sub-matrices containing that above considered.  Moreover the Gaussian row-reduction of $A_m$ with respect to the first $a_m$ rows is nothing else than the  $\mathcal{G}^{**}$-reduction of the $S$-polynomials  of the special type considered defining $\mathfrak{R}'$.
\end{proof}

  If we fix a term ordering $ \prec$, we can obtain the variety $\St(J,\prec)$ (parameterizing   all the homogeneous ideal such that $\In_{\prec}(I)=J$) as the  section of $\BSt(J)$ with the linear subspace $L$ given by the ideal $\left(c_{\alpha \gamma} \ / \ X^\alpha \prec X^\gamma \right)\subset k[C]$. If for every $m\leq m_0$ ($m_0$ as in the   Remark \ref{solom0}) $J_m$ is a $\prec$-segment (i.e. it is generated by the highest $\dim_k J_m$ monomials with respect to $\prec$), then $\St(J,\preceq)$ and $\BSt(J)$ are  isomorphic as affine schemes. In fact we can obtain both varieties using the same kind of construction. The only difference between the two cases is due to the set of monomials that can appear in the  tails  of the polynomials $F_\alpha$: every monomial in $\cN(J)$ of the same degree as $X^\alpha$ for $\BSt(J)$, only those that are $\prec$-lower than $X^\alpha$   for $\St(J,\prec)$.  

For some strongly stable ideals $J$ we can find a suitable term ordering such that $\St(J,\prec)= \BSt(J)$, but there are cases in which  $\bigcup_{\prec}\St(J,\prec)\subsetneq \BSt(J)$  (Example \ref{es5}).

The existence of a term ordering such that $ \BSt(J)=\St(J,\preceq)$ has interesting consequences on the geometrical features of the stratum. In fact the $\lambda$-grading on $k[C]$ is positive if and only if such a term ordering exists. In this case we can isomorphically project  $\St(J,\prec)= \BSt(J)$ in the Zariski tangent space at the origin (see \cite{FR}). As a consequence of this projection we can prove for instance that the stratum  is always connected;   moreover  it is isomorphic to an affine space, provided the origin is a smooth point. If  for a given ideal $J$ such a term ordering does not exist, then in general we cannot embed the stratum in the Zariski tangent space at the origin (Example \ref{es5}).  However  we do not know  examples of Borel ideals $J$ such that either $\BSt(J)$  has more than one connected component  or $J$ is smooth and $\BSt(J)$ is  not rational.

\begin{example}\label{es5}  Let $J$ be the strongly stable ideal $(x^3,x^2y,xy^2,y^5)_{\geq 5}$ in $k[x,y,z]$ (where we assume $x>y>z$). For every term ordering we can find in degree 5 a monomial in $J$ lower than  a monomial in $\cN (J)$, because  $xy^2z^2 \succ x^2z^3$ and $xy^2z^2\succ y^4z$ would be in contradiction with the equality  $(xy^2z^2)^2 =x^2z^3\cdot y^4z$.
The stratum $\BSt (J)$ is isomorphic to an open subset of the Hilbert scheme of 8 points in the projective plane, which is, as  well know,   irreducible of dimension 16 (see \cite{RoggBorelCovering}). It contains all the \Gr strata $\St(J,\prec)$ for every $\prec$ and also some more point, for instance the one corresponding to the ideal $I$ of Example  \ref{ripreso}.   Computing (using some computer system tool)   the Zariski tangent space to $\BSt(J)$ at the origin $T$, one can see that it has dimension 16, and so $J$ corresponds to a smooth point on it. However $\BSt(J)$ cannot be isomorphically projected on $T$, but only on a  linear space $T'\simeq \Af^{18}$  containing $T$. In this minimal embedding, $\BSt(J)$ is realized as an affine subscheme of $\Af^{18}$,   complete intersection of two hypersurfaces of degrees 6 and 7:

$F:=c_{11}c_{13}c_{12}^2-c_{13}^2c_{11}^2c_{16}^2+c_{11}c_{13}^2c_{5}+c_{13}^2c_{8}c_{6}-c_{13}^2c_{12}c_{6}+c_{13}c_{11}c_{15}-c_{13}^2c_{11}c_{9}+c_{11}^2c_{13}^2c_{8}^2+c_{18}c_{12}c_{11}^2-2c_{10}c_{11}c_{18}+c_{8}c_{11}^2c_{18}-c_{8}^2-c_{11}c_{16}c_{13}^3c_{6}+c_{13}c_{18}c_{11}^3c_{16}+c_{10}c_{13}^3c_{6}-c_{10}c_{13}c_{8}-c_{11}c_{10}c_{13}^2c_{8}+c_{13}c_{18}c_{12}c_{11}^3-c_{11}c_{6}c_{12}c_{13}^3+c_{13}^2c_{6}c_{16}+c_{18}c_{11}^2c_{16}+c_{11}c_{7}c_{13}^2+c_{13}c_{5}-c_{11}c_{6}c_{13}^3c_{8}+c_{12}c_{8}+(-1-2c_{11}c_{13})c_{1}+c_{13}c_{11}^2c_{17}+c_{11}^2c_{8}c_{12}c_{13}^2-c_{6}c_{18}-c_{13}^2c_{11}^2c_{16}c_{12}-2c_{13}c_{10}c_{11}^2c_{18}+c_{13}c_{8}c_{11}^3c_{18}-c_{7}c_{13}+2c_{13}^2c_{10}c_{11}c_{16}+2c_{11}c_{8}c_{12}c_{13}+c_{13}c_{9}-c_{13}^2c_{10}^2-c_{13}^3c_{3}$

$ G:=c_{16}c_{13}c_{5}+c_{16}c_{13}c_{9}+c_{8}c_{6}c_{18}-c_{13}^2c_{11}^2c_{16}^3-c_{16}c_{7}c_{13}+c_{16}c_{12}c_{8}-c_{16}c_{13}^2c_{10}^2-c_{8}c_{13}c_{5}+c_{13}^2c_{6}c_{16}^2-c_{8}c_{13}c_{9}-2c_{8}^2c_{11}^2c_{18}-c_{12}c_{13}c_{5}-c_{11}^2c_{13}^2c_{8}^3+c_{10}c_{13}c_{8}^2+c_{8}c_{13}^2c_{10}^2-c_{16}c_{6}c_{18}+c_{18}c_{11}^2c_{16}^2-c_{6}^2c_{13}^4c_{8}+c_{12}c_{6}c_{18}+c_{11}c_{18}c_{9}+c_{12}c_{7}c_{13}-c_{13}^2c_{7}c_{10}-c_{16}c_{8}^2-c_{13}c_{11}c_{17}c_{10}+c_{13}c_{11}^2c_{17}c_{12}-c_{13}^4c_{16}c_{6}^2-c_{12}c_{13}^4c_{6}^2-c_{6}c_{18}^2c_{11}^2+c_{8}c_{13}^2c_{11}^2c_{16}^2+c_{8}c_{13}c_{11}^2c_{17}+c_{11}c_{8}c_{18}c_{10}+c_{13}c_{6}c_{18}c_{10}-c_{11}c_{17}c_{13}^2c_{6}+c_{12}c_{11}c_{13}^2c_{5}+c_{16}c_{11}c_{7}c_{13}^2-2c_{12}^2c_{13}^2c_{6}-c_{11}c_{18}c_{5}+c_{15}c_{11}^2c_{18}-c_{11}c_{17}c_{8}-c_{13}^2c_{8}^2c_{6}+c_{11}^3c_{17}c_{18}-c_{16}c_{10}c_{13}c_{8}-c_{13}c_{11}^2c_{18}c_{9}-c_{12}^2c_{11}c_{13}^3c_{6}+2c_{13}^2c_{9}c_{10}+c_{13}c_{11}^2c_{18}c_{7}+c_{16}c_{13}c_{11}^2c_{17}+(2c_{13}^2c_{6}-c_{16}-3c_{11}^2c_{18}+c_{8})c_{1}-c_{13}c_{11}^3c_{18}c_{16}^2-c_{13}^2c_{11}^2c_{16}c_{12}^2+c_{8}c_{18}c_{12}c_{11}^2+c_{7}c_{11}c_{13}^2c_{8}-c_{13}c_{10}^2c_{11}c_{18}-2c_{11}c_{16}c_{13}^2c_{9}-2c_{11}c_{16}^2c_{13}^3c_{6}-2c_{11}c_{8}^2c_{12}c_{13}+c_{13}c_{11}^3c_{12}^2c_{18}+2c_{12}c_{10}c_{11}c_{18}-3c_{11}c_{18}c_{13}^2c_{3}-c_{16}c_{13}^2c_{11}^2c_{6}c_{18}-c_{12}c_{13}^2c_{11}c_{9}+c_{11}^2c_{12}^2c_{8}c_{13}^2+c_{13}c_{18}c_{3}-3c_{16}c_{11}c_{6}c_{12}c_{13}^3+c_{12}c_{10}c_{13}^3c_{6}+3c_{13}c_{11}^2c_{18}c_{5}-c_{15}c_{13}^2c_{6}+2c_{8}^2c_{11}c_{13}^3c_{6}+c_{8}c_{11}c_{6}c_{12}c_{13}^3-3c_{13}c_{11}c_{18}c_{12}c_{6}+c_{16}c_{11}^2c_{13}^2c_{8}^2-c_{13}^2c_{8}c_{10}c_{12}c_{11}+2c_{11}^4c_{18}^2c_{12}+2c_{13}c_{10}c_{11}^2c_{18}c_{16}+2c_{12}c_{13}^2c_{10}c_{11}c_{16}+(-1+2c_{13}c_{11})c_{2}-2c_{13}^3c_{9}c_{6}-c_{13}^2c_{8}c_{11}^2c_{6}c_{18}+c_{13}c_{11}c_{18}c_{6}c_{16}-2c_{16}c_{10}c_{11}c_{18}+2c_{8}c_{13}^2c_{11}^2c_{16}c_{12}+c_{10}c_{11}c_{18}c_{13}^2c_{6}-c_{12}c_{13}^3c_{3}-2c_{11}c_{16}c_{13}^2c_{8}c_{10}+c_{8}^3-c_{12}^2c_{8}-c_{13}c_{11}c_{14}+c_{11}c_{6}c_{18}c_{13}c_{8}+c_{7}c_{11}c_{12}c_{13}^2+c_{12}c_{10}c_{13}c_{8}+2c_{11}^4c_{18}^2c_{16}-c_{13}c_{8}c_{11}^2c_{10}c_{18}-2c_{16}^2c_{12}c_{11}^2c_{13}^2+c_{13}c_{8}^2c_{11}^3c_{18}+2c_{13}c_{12}c_{8}c_{11}^3c_{18}+2c_{11}^4c_{18}^2c_{8}-c_{12}c_{11}^2c_{18}c_{13}^2c_{6}-4c_{11}^3c_{18}^2c_{10}+c_{16}c_{13}c_{11}c_{15}+c_{4}c_{13}^2+2c_{8}c_{11}c_{13}^2c_{9}-c_{6}c_{7}c_{13}^3-c_{13}c_{12}c_{9}-2c_{13}c_{11}^2c_{18}c_{10}c_{12}+2c_{16}c_{10}c_{13}^3c_{6}+2c_{13}^2c_{10}c_{11}c_{16}^2.$

   A MAPLE session with an explicit computation of equations defining $\BSt(J)$ in its minimal embedding can be found at:\ \texttt{http://www2.dm.unito.it/paginepersonali/roggero/Stratum1/}

\end{example}
\bibliographystyle{plain}

\end{document}